\documentclass[preprint,11pt]{elsarticle}
\topskip  \usepackage{amssymb}
\usepackage{amsmath}
\usepackage{amsthm}
\usepackage{hyperref}
\newtheorem{definition}{Definition}
\newtheorem{theorem}{Theorem}
\newtheorem{lemma}{Lemma}
\newtheorem{corollary}{Corollary}
\newtheorem{remark}{Remark}
\numberwithin{equation}{section} 
\numberwithin{definition}{section}
\numberwithin{theorem}{section}
\numberwithin{lemma}{section}
\numberwithin{corollary}{section}
\numberwithin{remark}{section}
\begin{document}
\title{\bf{An extension of Apostol type of Hermite-Genocchi polynomials and their probabilistic representation}}
\author {Beih S. El-Desouky,  \footnote{Corresponding author: b\_desouky@yahoo.com }Rabab S. Gomaa and Alia M. Magar}
\address{Department of Mathematics, Faculty of Science, \\Mansoura University, 35516 Mansoura, Egypt}
\begin{abstract}
The main purpose of this paper is to introduce and investigate the various properties of a new generalization of Apostol Hermite - Genocchi polynomials. We derive many useful results involving new generalized Apostol Hermite - Genocchi polynomials. We also consider some statistical application of the new family in probability distribution theory and reliability.\\


\end{abstract}
\maketitle
\section{Introduction}
The generalized Apostol-Bernoulli polynomials $ B_{n}^{(\alpha)}(x;\lambda)$ of order $\alpha\in \mathbb{C},$ the generalized Apostol-Euler polynomials
$ E_{n}^{(\alpha)}(x;\lambda)$ of order $\alpha\in \mathbb{C}$ and the generalized Apostol-Genocchi polynomials
$ G_{n}^{(\alpha)}(x;\lambda)$ of order $\alpha\in \mathbb{C}$ are defined, see (\cite{lu1}, \cite{lu2}, \cite{lu3} and \cite{lu5}) respectively, through the generating function by:
\begin{equation*}
 \left( \frac{t}{\lambda e^{t}-1}\right)^{\alpha}\;e^{xt}=
 \sum_{n=0}^{\infty}B_{n}^{^{(\alpha)}}(x;\lambda)\frac{t^{n}}{n!},
\end{equation*}
\begin{equation}\label{1}
 (\mid t\mid <2\pi,\;\; \text{when}\;\; \lambda=1;\;\; \mid t\mid <\mid\log \lambda\mid,\;\; \text{when}\;\; \lambda\neq 1,\;\;1^{\alpha}:=1)
\end{equation}
\begin{equation*}
 \left( \frac{2}{\lambda e^{t}+1}\right)^{\alpha}\;e^{xt}=
 \sum_{n=0}^{\infty}E_{n}^{^{(\alpha)}}(x;\lambda)\frac{t^{n}}{n!},
\end{equation*}
\begin{equation}\label{2}
 (\mid t\mid <\pi,\;\; \text{when}\;\; \lambda=1;\;\; \mid t\mid <\mid\log (-\lambda)\mid,\;\; \text{when}\;\; \lambda\neq 1,\;\;1^{\alpha}:=1)
\end{equation}
and
\begin{equation*}
 \left( \frac{2t}{\lambda e^{t}+1}\right)^{\alpha}\;e^{xt}=
 \sum_{n=0}^{\infty}G_{n}^{^{(\alpha)}}(x;\lambda)\frac{t^{n}}{n!},
\end{equation*}
\begin{equation}\label{3}
 (\mid t\mid <\pi,\;\; \text{when}\;\; \lambda=1;\;\; \mid t\mid <\mid\log(- \lambda)\mid,\;\; \text{when}\;\; \lambda\neq 1,\;\;1^{\alpha}:=1)
\end{equation}
Gaboury and Kurt \cite{kurt} gave a new class of the generalized Apostol Hermite-Genocchi polynomials by means of the following generating function:
\begin{equation}\label{4}
\left( \frac{2t}{\lambda e^{t}+1}\right)^{\alpha}\;e^{xt+yt^{2}}=
 \sum_{n=0}^{\infty}{}_{H}G_{n}^{^{(\alpha)}}(x,y;\lambda)\frac{t^{n}}{n!}.
\end{equation}
In the case $ \alpha=1$, it is reduced to Apostol Hermite-Genocchi polynomials defined by Dattoli et al. \cite{datt} as follows:
\begin{equation}\label{5}
 \frac{2t}{e^{t}+1}\;e^{xt+yt^{2}}=
 \sum_{n=0}^{\infty}{}_{H}G_{n}(x,y)\frac{t^{n}}{n!}.
\end{equation}
The generalized Apostol-Genocchi polynomials with the parameters $a,\; b$ and $c$ are defined by Jolany et al. \cite{jolany} in the following form:
\begin{equation}\label{6}
 \sum_{n=0}^{\infty} G_{n}(x;a,b,c;\lambda)\frac{t^{n}}{n!}=\frac{2t}{\lambda b^{t}+a^{t}}\;\;c^{xt},
\end{equation}
where $a,\;b$ and $c$ are positive integers.\\
For a real or complex parameter $\alpha,$ the generalized Apostol-Genocchi polynomials $ G_{n}^{^{(\alpha)}}(x;a,b,c;\lambda)$ of order $\alpha$ with parameters  $a,\; b $ and $c$ are defined by
\begin{equation}\label{7}
 \sum_{n=0}^{\infty} G_{n}^{^{(\alpha)}}(x;a,b,c;\lambda)\frac{t^{n}}{n!}=\left(\frac{2t}{\lambda b^{t}+a^{t}}\right)^{\alpha}\;\;c^{xt}.
\end{equation}
Gaboury and Kurt \cite{kurt} gave the generalization of the Apostol Hermite-Genocchi polynomials with parameters $a,\; b $ and $c$, as follows:
 \begin{equation*}
   \sum_{n=0}^{\infty} {}_{H}G_{n}^{^{(\alpha)}}(x,y;a,b,c;\lambda)\frac{t^{n}}{n!}=\left(\frac{2t}{\lambda b^{t}+a^{t}}\right)^{\alpha}\;\;c^{xt+yt^{2}}.
 \end{equation*}
 \begin{equation}\label{8}
    \left( \mid t\mid <\mid \frac{\log(-\lambda)}{\log(\frac{b}{a})}\mid;\,\,a\in\mathbb{C}\backslash \{0\},\; b,\; c\in \mathbb{R}^{+};\,\,1^{\alpha}:=1\right).
 \end{equation}
For a real or complex parameter $\alpha,$ the generalized Apostol Hermite-Genocchi polynomials $ {}_{H}G_{n}^{^{(\alpha)}}(x;a,b,c;\lambda)$ of order $\alpha$ with parameters $a\,, b $ are defined by, \cite{kurt}
 \begin{equation*}
   \sum_{n=0}^{\infty} {}_{H}G_{n}^{^{(\alpha)}}(x,y;a,b,e;\lambda)\frac{t^{n}}{n!}=\left(\frac{2t}{\lambda b^{t}+a^{t}}\right)^{\alpha}\;\;e^{xt+yt^{2}}.
 \end{equation*}
 \begin{equation}\label{l}
    \left( \mid t\mid <\mid \frac{\log(-\lambda)}{\log(\frac{b}{a})}\mid;\,\,a\in\mathbb{C}\backslash \{0\},\; b\in \mathbb{R}^{+};\,\,1^{\alpha}:=1\right).
 \end{equation}

Araci et al. \cite{araci} introduced a new concept of the Apostol Hermite-Genocchi polynomials by means of the following generating function:
\begin{equation*}
  \sum_{n=0}^{\infty} G_{n}^{^{(\alpha)}}(x,y;a,b,c;\lambda)\frac{t^{n}}{n!}=\left(\frac{2t}{\lambda b^{t}+a^{t}}\right)^{\alpha}\;\;c^{xt+h(t,y)}.
\end{equation*}
\begin{equation}\label{9}
 \left( \mid t\mid <\mid \frac{\log(-\lambda)}{\log(\frac{b}{a})}\mid;\,\,a\in\mathbb{C}\backslash \{0\},\; b,\;c\in \mathbb{R}^{+};\,\,1^{\alpha}:=1\right).
\end{equation}

Let $c$ be positive integer. The generalized 2-variable 1-parameter Hermite Kamp'e de Feriet polynomials $H_{n}(x,y,c)$ for nonnegative integer $n$ is defined by (see \cite{bell}, \cite{patha} and \cite{pathb})
\begin{equation}\label{H}
\sum_{n=0}^{\infty}{H}_{n}(x,y,c)\frac{t^{n}}{n!}=c^{xt+yt^{2}},
\end{equation}
which is an extention of 2-variable Hermite Kamp'e de Feriet polynomials $H_{n}(x,y)$ defind by
\begin{equation}\label{H2}
\sum_{n=0}^{\infty}{H}_{n}(x,y)\frac{t^{n}}{n!}=e^{xt+yt^{2}}.
\end{equation}

In 1975, Nakagawa and Osaki \cite{Naka} were the first to study a discrete life time distribution
which is defined as the discrete counterpart of the usual continuous Weibull distribution.
In 1982, Salvia and Bollinger \cite{Salvia} introduced basic results about discrete reliability and illustrated them with the simple discrete life distributions. The characterization of discrete distributions has been studied by Roy, Gupta, Gupta in 1999 \cite{Roy}. In 1997, Gupta, Gupta and Tripahi \cite{Gupta} introduced wide classes of discrete distributions with increasing failure
rate. In 1997, Nair and Asha \cite{Nair} introduced some classes of multivariate life distributions
in discrete time.

\section{A new generalization of the Apostol Hermite-Genocchi polynomials}
\begin{definition}\hfill \\
Let $a, b$ and $c$ be positive integers with the condition $a\neq b$. A new generalization of the Apostol Hermite-Genocchi polynomials
$\mathbb{}_{H}{M}_{n}^{(r)}(x,y;a,b,c;\bar\alpha_{r})$ for nonnegative integer $n$ is defined by means by the generating function
\begin{equation*}
\sum_{n=0}^{\infty}\mathbb{}_{H}{M}_{n}^{^{(r)}}(x,y;k,a,b,c;\bar\alpha_{r})\frac{t^{n}}{n!}=\frac{(-1)^{r}t^{rk}2^{r(1-k)}}{\prod\limits^{r-1}_{i=0}(\alpha_{i}b^{t}-a^{t})} \,\,c^{xt+h(t,y)},
\end{equation*}
\begin{equation}\label{10}
 \left( \mid t\mid <\mid \frac{\log(\alpha_{i})}{\log(\frac{b}{a})}\mid;\,\, a,b,c\in \mathbb{R}^{+};\,\, \alpha_{i}\neq1;\,\, \forall\,\, i=0,1,\cdots,r-1\right),
\end{equation}
where $r\in\mathbb{C};\,\, \bar\alpha_{r}=(\alpha_{0},\alpha_{1},\cdots,\alpha_{r-1})$ is a sequence of complex numbers.
\end{definition}

Setting $ h(t,y)=yt^{2}$ in \eqref{10}, we get the following definition.

\begin{definition}\hfill\\
Let $a, b$ and $c$ be positive integers with the condition $a\neq b$. A new generalization of the Apostol Hermite-Genocchi polynomials
$\mathbb{}_{H}{M}_{n}^{^{(r)}}(x,y;a,b,c;\bar\alpha_{r})$ for nonnegative integer $n$ is defined by means by the generating function
\begin{equation*}
\sum_{n=0}^{\infty}\mathbb{}_{H}{M}_{n}^{(r)}(x,y;a,b,c;\bar\alpha_{r})\frac{t^{n}}{n!}=\frac{(-1)^{r}t^{rk}2^{r(1-k)}}{\prod\limits^{r-1}_{i=0}(\alpha_{i}b^{t}-a^{t})} \,\,c^{xt+yt^{2}},
\end{equation*}
\begin{equation}\label{11}
\left( \mid t\mid <\mid \frac{\log(\alpha_{i})}{\log(\frac{b}{a})}\mid;\,\, a,b,c\in \mathbb{R}^{+};\,\, \alpha_{i}\neq1;\,\, \forall\,\, i=0,1,\cdots,r-1\right),
\end{equation}
where $r\in\mathbb{C};\,\, \bar\alpha_{r}=(\alpha_{0},\alpha_{1},\cdots,\alpha_{r-1})$ is a sequence of complex numbers.
\end{definition}
\begin{remark}
If we set $x=y=0$ in \eqref{11}, then we obtain the new unified generalized Apostol Hermite-Genocchi numbers, as\\
\center{$ \mathbb{}_{H}{M}_{n}^{(r)}(0,0;a,b,1;\bar\alpha_{r})=\mathbb{}_{H}{M}_{n}^{(r)}(a,b;\bar\alpha_{r}).$}
\end{remark}
\subsection{Special cases}
\begin{enumerate}
  \item  setting $\alpha_{i}=-\lambda \,\,and \,\, k=1 \,\,in \eqref{10},$ we get
\begin{center}
$\mathbb{}_{H}{M}_{n}^{^{(r)}}(x,y;1,a,b,c;-\lambda)=2^{-r}\;G_{n}^{^{(r)}}(x,y;a,b,c;\lambda).$
\end{center}
(Generalized Apostol-Genocchi polynomials,\; see \cite{araci} )
  \item  setting $\alpha_{i}=-\lambda\,\, and \,\, k=1\,\, in \eqref{11},$ we get
\begin{center}
  $\mathbb{}_{H}{M}_{n}^{^{(r)}}(x,y;a,b,c;-\lambda)=2^{-r}\; _{H}G_{n}^{^{(r)}}(x,y;a,b,c;\lambda).$
\end{center}
(Generalized Apostol Hermite-Genocchi polynomials of order $r$, see \cite{kurt})
  \item  setting $\alpha_{i}=-\lambda, \,\, k=1\,\, and \,\,c=e\,\, in \eqref{11},$ we have
\begin{center}
  $\mathbb{}_{H}{M}_{n}^{^{(r)}}(x,y;a,b,e;-\lambda)=2^{-r}\; _{H}G_{n}^{^{(r)}}(x,y;a,b;\lambda).$
\end{center}
(Generalized Apostol Hermite-Genocchi polynomials with parameter $a$, $b$ and $c$, see \cite{kurt})
  \item  setting $\alpha_{i}=\lambda \,\,and \,\, k=1 \,\,in \eqref{11},$ we get
\begin{center}
  $\mathbb{}_{H}{M}_{n}^{^{(r)}}(x,y;a,b,c;\lambda)=(-1)^{r} B_{n}^{^{(r)}}(x,y;a,b,c;\lambda).$
\end{center}
(Generalized Apostol-Bernoulli polynomials of order r,\;see \cite{sriv}).
  \item setting $\alpha_{i}=-\lambda\,\, and \,\, k=0 \,\,in \eqref{11},$ we get
\begin{center}
  $\mathbb{}_{H}{M}_{n}^{^{(r)}}(x,y;a,b,c;-\lambda)=E_{n}^{(r)}(x,y;a,b,c;\lambda).$
\end{center}
(Generalized Apostol-Euler polynomials of order r,\; see \cite{sriv}).
   \item  setting $c=b=e,\; y=0,\; and \,\,a=1\,\, in \eqref{11},$ we get
\begin{center}
  $\mathbb{}_{H}{M}_{n}^{^{(r)}}(x,y;1,e,e;\bar{\alpha_{r}})=(-1)^{-r}M_{n}^{^{(r)}}(x;k,\bar{\alpha_{r}}).$
\end{center}
(Unified family of generalized Apostol-Bernoulli, Euler and Genocchi polynomials, see \cite{bs}).
   \item  setting $\alpha_{i}=-\lambda ,\, k=1,c=b=e\,\, and \,\,a=1\,\, in \eqref{11},$ we get
\begin{center}
  $\mathbb{}_{H}{M}_{n}^{^{(r)}}(x,y;1,e,e;-\lambda)=(-2)^{r} _{H}G_{n}^{^{(r)}}(x,y;\lambda).$
\end{center}
(Generalized Apostol Hermite-Genocchi polynomials,\; see \cite{kurt}).
  \item  setting $\alpha_{i}=\lambda ,\, k=1,c=b=e\,\, and \,\,a=1\,\, in \eqref{11},$ we get
\begin{center}
  $\mathbb{}_{H}{M}_{n}^{^{(r)}}(x,y;1,e,e;\lambda)=(-1)^{r} _{H}B_{n}^{^{(r)}}(x,y;\lambda).$
\end{center}
(Hermite-based generalized Apostol-Bernoulli polynomials,\;see \cite{kurt}).
\end{enumerate}
\begin{theorem}
Let $a,\;b $ and $c$ be positive integers with the rule $a\neq b.$ For $x\in \mathbf{R}$ and $n\geq 0.$ Then we have
\begin{equation}\label{22}
 \mathbb{}_{H}{M}_{n}^{(r+\beta)}(x+y,z+u;k,a,b,c;\bar{\alpha_{r}})=\sum_{k=0}^{\infty}\binom{n}{k} \mathbb{}_{H}{M}_{n}^{(r)}(y,z;k,a,b,c;\bar{\alpha_{r}}) \mathbb{}_{H}{M}_{n}^{(\beta)}(x,u;k,a,b,c;\bar{\alpha_{r}}).
\end{equation}
\end{theorem}
\begin{proof}
From \eqref{11}, we get
\begin{equation*}
   \sum_{n=0}^{\infty} \mathbb{}_{H}{M}_{n}^{(r+\beta)}(x+y,z+u;k,a,b,c;\bar{\alpha_{r}}) =\frac{(-1)^{r+\beta}2^{(r+\beta)(1-k)}t^{(r+\beta)k}}{\prod\limits_{i=0}^{(r+\beta)-1}(\alpha_{i}b^{t}-a^{t})}c^{(x+y)t+(z+u)t^{2}}
\end{equation*}

\begin{equation*}
 \;\;\;\;\;\;\;\;\;\;\;\;\;\;\;\;\;\;\;\;\;\;\;\;\;\;\;\;\;\;\;\;\;\;\;\;\;\;\;\;\;\;\;\;= \left( \sum_{n=0}^{\infty} \mathbb{}_{H}{M}_{n}^{(r)}(y,z;k,a,b,c;\bar{\alpha_{r}})\frac{t^{n}}{n!}\right)\left( \sum_{n=0}^{\infty} \mathbb{}_{H}{M}_{n}^{(\beta)}(x,u;k,a,b,c;\bar{\alpha_{r}})\frac{t^{n}}{n!}\right).
\end{equation*}

By using Cauchy product and comparing the coefficients of $t^{n}$ on both sides, yields \eqref{22}.
\end{proof}
\begin{theorem}\hfill\\
Let a, b and c be positive integers with the rule $a\neq b$ for $x\in\mathbf{R}$ and $n \geq 0$. Then we have:
\begin{equation}\label{12}
 \mathbb{}_{H}{M}_{n}^{(r)}(x+z,y;k,a,b,c;\bar{\alpha})=\sum_{m=0}^{n}\binom{n}{m} \mathbb{}_{H}{M}_{n-m}^{(r)}(z;k,a,b,c;\bar{\alpha}) H_{m}(x,y,c).
\end{equation}
\end{theorem}
\begin{proof}
From \eqref{11}, we have
\begin{eqnarray*}
  \sum_{n=0}^{\infty}\mathbb{}_{H}{M}_{n}^{(r)}(x+z,y;k,a,b,c;\bar{\alpha})\frac{t^{n}}{n!} &=& \frac{(-1)^{r}2^{r(1-k)}t^{rk}}{\prod\limits_{i=0}^{r-1}(\alpha_{i}b^{t}-a^{t})}c^{(x+z)t+yt^{2}} \\
 &=& \sum_{n=0}^{\infty}\mathbb{}_{H}{}_{H}{M}_{n}^{(r)}(z;k,a,b,c;\bar{\alpha})\frac{t^{n}}{n!}\sum_{m=0}^{\infty}H_{m}(x,y,c)\frac{t^{m}}{m!},
\end{eqnarray*}
using Cauchy product and equating the coefficients of $t^{n}$ on both sides of the last equation, yields \eqref{12}.
\end{proof}
\section{Implicit summation formulas on the  generalized Apostol Hermite-Genocchi polynomials }
\begin{theorem}\hfill\\
Let a, b and c be positive integers with $a\neq b$ for $x,\; y\in\mathbf{R}$ and $n \geq 0$. Then we have:
\begin{equation}\label{120}
 \mathbb{}_{H}{M}_{n}^{(r)}(x+z,y;k,a,b,c;\bar{\alpha})=\sum_{l=0}^{n}\binom{n}{l} z^{n-l}(\ln c)^{n-l}\mathbb{}_{H}{M}_{n}^{(r)}(x,y;k,a,b,c;\bar{\alpha}).
\end{equation}
\end{theorem}
\begin{proof}
From Eq. \eqref{11}, we have
\begin{eqnarray*}
  \sum_{n=0}^{\infty}\mathbb{}_{H}{M}_{n}^{(r)}(x+z,y;k,a,b,c;\bar{\alpha})\frac{t^{n}}{n!} &=& \frac{(-1)^{r}2^{r(1-k)}t^{rk}}{\prod\limits_{i=0}^{r-1}(\alpha_{i}b^{t}-a^{t})}c^{(x+z)t+yt^{2}} \\
 &=& \sum_{i=0}^{\infty}\mathbb{}_{H}{M}_{i}^{(r)}(x,y;k,a,b,c;\bar{\alpha})\frac{t^{i}}{i!}\sum_{l=0}^{\infty}\frac{(z\;t\; \log c)^{l}}{l!},
\end{eqnarray*}
using Cauchy product and equating the coefficients of $t^{n}$ on both sides of the last equation, yields \eqref{120}.
\end{proof}

\begin{theorem}\hfill\\
Let $a,\;b$ and $c$ positive integers, by $a\neq b$ then, for $ x,\;y \in \mathbf{R}$ and $n,\;m \geq0,$ we have
\begin{equation}\label{14}
\mathbb{}_{H}{M}_{n+m}^{(r)}(z,y;k,a,b,c;\bar{\alpha_{r}})=\sum_{s=0}^{m}\sum_{\ell=0}^{n}\binom{m}{s}\binom{n}{\ell}(\log c)^{s+\ell}(z-x)^{s+\ell}\;\mathbb{}_{H}{M}_{n+m-s-\ell}^{(r)}(x,y;k,a,b,c;\bar{\alpha_{r}}).
\end{equation}
\end{theorem}
\begin{proof}
 Replacing $t$ by $t+u$ and rewrite the generating function \eqref{11} as the following
\begin{equation*}
  \frac{(-1)^{r}(t+u)^{rk}2^{r(1-k)}}{\prod\limits_{i=0}^{r-1}(\alpha_{i}b^{t+u}-a^{t+u})}\; c^{x(t+u)+y(t+u)^{2}}=
  \sum_{n=0}^{\infty}\sum_{m=0}^{\infty}\mathbb{}_{H}{M}_{n+m}^{(r)}(x,y;k,a,b,c;\bar{\alpha_{r}})\frac{t^{n}}{n!}\frac{u^{m}}{m!}
\end{equation*}
\begin{equation}\label{e}
    \frac{(-1)^{r}(t+u)^{rk}2^{r(1-k)}}{\prod\limits_{i=0}^{r-1}(\alpha_{i}b^{t+u}-a^{t+u})}\;c^{y(t+u)^{2}}= c^{-x(t+u)}  \sum_{n=0}^{\infty}\sum_{m=0}^{\infty}\mathbb{}_{H}{M}_{n+m}^{(r)}(x,y;k,a,b,c;\bar{\alpha_{r}})\frac{t^{n}}{n!}\frac{u^{m}}{m!}.
\end{equation}
Replacing $x$ by $z$ in Eq. \eqref{e}, we have
\begin{equation*}
  c^{-x(t+u)}  \sum_{n=0}^{\infty}\sum_{m=0}^{\infty}\mathbb{}_{H}{M}_{n+m}^{(r)}(x,y;k,a,b,c;\bar{\alpha_{r}})\frac{t^{n}}{n!}\frac{u^{m}}{m!}=
  c^{-z(t+u)}  \sum_{n=0}^{\infty}\sum_{m=0}^{\infty}\mathbb{}_{H}{M}_{n+m}^{(r)}(z,y;k,a,b,c;\bar{\alpha_{r}})\frac{t^{n}}{n!}\frac{u^{m}}{m!}
\end{equation*}
\begin{equation}\label{e2}
  c^{(z-x)(t+u)}  \sum_{n=0}^{\infty}\sum_{m=0}^{\infty}\mathbb{}_{H}{M}_{n+m}^{(r)}(x,y;k,a,b,c;\bar{\alpha_{r}})\frac{t^{n}}{n!}\frac{u^{m}}{m!}= \sum_{n=0}^{\infty}\sum_{m=0}^{\infty}\mathbb{}_{H}{M}_{n+m}^{(r)}(z,y;k,a,b,c;\bar{\alpha_{r}})\frac{t^{n}}{n!}\frac{u^{m}}{m!}.
\end{equation}
By applying [see Pathan and khan \cite{pathd}, p. 52]
\begin{equation*}
  \sum_{N=0}^{\infty}f(N)\frac{(x+y)^{N}}{N!}= \sum_{n=0}^{\infty}\sum_{m=0}^{\infty}f(n+m)\frac{x^{n}}{n!}\frac{y^{m}}{m!}
\end{equation*}
to $c^{(z-x)(t+u)}$ in Eq. \eqref{e2}, we get
\begin{equation*}
 \sum_{N=0}^{\infty} (\log c)^{N} \frac{[(z-x)(t+u)]^{N}}{N!} \sum_{n=0}^{\infty}\sum_{m=0}^{\infty}\mathbb{}_{H}{M}_{n+m}^{(r)}(x,y;k,a,b,c;\bar{\alpha_{r}})\frac{t^{n}}{n!}\frac{u^{m}}{m!}
\end{equation*}
\begin{equation*}
  = \sum_{n=0}^{\infty}\sum_{m=0}^{\infty}\mathbb{}_{H}{M}_{n+m}^{(r)}(z,y;k,a,b,c;\bar{\alpha_{r}})\frac{t^{n}}{n!}\frac{u^{m}}{m!}
\end{equation*}
\begin{equation*}
 \sum_{\ell=0}^{\infty}\sum_{s=0}^{\infty}\frac{(\log c)^{s+\ell} (z-x)^{s+\ell}}{\ell! s!} \; t^{\ell}\;u^{s} \sum_{n=0}^{\infty}\sum_{m=0}^{\infty}\mathbb{}_{H}{M}_{n+m}^{(r)}(x,y;k,a,b,c;\bar{\alpha_{r}})\frac{t^{n}}{n!}\frac{u^{m}}{m!}
\end{equation*}
\begin{equation}\label{e3} =\sum_{n=0}^{\infty}\sum_{m=0}^{\infty}\mathbb{}_{H}{M}_{n+m}^{(r)}(z,y;k,a,b,c;\bar{\alpha_{r}})\frac{t^{n}}{n!}\frac{u^{m}}{m!}.
\end{equation}
Replacing $n$ by $n-\ell$ and $m$ by $m-s$ in Eq. \eqref{e3}, we have
\begin{equation*}
 \sum_{\ell=0}^{\infty}\sum_{s=0}^{\infty}\frac{(\log c)^{s+\ell} (z-x)^{s+\ell}}{\ell! s!} \; t^{\ell}\;u^{s} \sum_{n=\ell}^{\infty}\sum_{m=s}^{\infty}\mathbb{}_{H}{M}_{n+m-\ell-s}^{(r)}(x,y;k,a,b,c;\bar{\alpha_{r}})\frac{t^{n-\ell}}{(n-\ell)!}\frac{u^{m-s}}{(m-s)!}
\end{equation*}
\begin{equation*}
  = \sum_{n=0}^{\infty}\sum_{m=0}^{\infty}\mathbb{}_{H}{M}_{n+m}^{(r)}(z,y;k,a,b,c;\bar{\alpha_{r}})\frac{t^{n}}{n!}\frac{u^{m}}{m!}.
\end{equation*}
By using the lemma in [Srivastava \cite{sriv}, p. 100], we get
\begin{equation*}
  \sum_{n=0}^{\infty}\sum_{m=0}^{\infty} \left(\sum_{\ell=0}^{\infty}\sum_{s=0}^{\infty}\frac{(\log c)^{s+\ell} (z-x)^{s+\ell}}{\ell! s!} \; \mathbb{}_{H}{M}_{n+m-\ell-s}^{(r)}(x,y;k,a,b,c;\bar{\alpha_{r}})\right)\frac{t^{n}}{(n-\ell)!}\frac{u^{m}}{(m-s)!}
\end{equation*}
\begin{equation*}
   = \sum_{n=0}^{\infty}\sum_{m=0}^{\infty}\mathbb{}_{H}{M}_{n+m}^{(r)}(z,y;k,a,b,c;\bar{\alpha_{r}})\frac{t^{n}}{n!}\frac{u^{m}}{m!}.
\end{equation*}
Comparing the coefficients of $t^{n}\;u^{m}$, yields \eqref{14}.
\end{proof}

\begin{theorem}\hfill\\
Let a, b and c be positive integers with the rule $a\neq b$ for $x,\;y\in\mathbf{R}$ and $n \geq 0$, we have:
\begin{equation}\label{121}
 \mathbb{}_{H}{M}_{n}^{(r)}(x,y;k,a,b,c;\bar{\alpha})=\sum_{m=0}^{n}\binom{n}{m} \mathbb{}_{H}{M}_{n-m}^{(r)}(a,b;\bar{\alpha}) H_{m}(x,y,c).
\end{equation}
\end{theorem}
\begin{proof}
From Eq. \eqref{11}, we have
\begin{eqnarray*}
  \sum_{n=0}^{\infty}\mathbb{}_{H}{M}_{n}^{(r)}(x,y;k,a,b,c;\bar{\alpha})\frac{t^{n}}{n!} &=& \frac{(-1)^{r}2^{r(1-k)}t^{rk}}{\prod\limits_{i=0}^{r-1}(\alpha_{i}b^{t}-a^{t})}c^{(x)t+yt^{2}} \\
 &=&\sum^{\infty}_{n=0}\left( \sum_{m=0}^{n} \binom{n}{m}\mathbb{}_{H}{M}_{n-m}^{(r)}(a,b;\bar{\alpha})H_{m}(x,y,c)\right)\frac{t^{n}}{n!}.
\end{eqnarray*}
Thus, equating the coefficients of $t^{n}$ on both sides of last equation, yields \eqref{121}.
\end{proof}

\begin{theorem}\hfill\\
For arbitrary real or complex parameter (r), the following implicit summation formula holds true
\begin{equation}\label{19}
  \mathbb{}_{H}{M}_{n}^{(r)}(x+1,y;k,a,b,c;\bar{\alpha_{r}})=\sum_{k=0}^{n}\binom{n}{k}(\log c)^{n-k}\;\mathbb{}_{H}{M}_{n}^{(r)}(x,y;k,a,b,c;\bar{\alpha_{r}}).
\end{equation}
\end{theorem}
\begin{proof}
From \eqref{11}, we have
\begin{eqnarray*}
  \sum_{n=0}^{\infty}\mathbb{}_{H}{M}_{n}^{(r)}(x+1,y;k,a,b,c;\bar{\alpha_{r}})\frac{t^{n}}{n!} &=& \frac{(-1)^{r}2^{r(1-k)}t^{rk}}{\prod\limits_{i=0}^{r-1}(\alpha_{i}b^{t}-a^{t})}c^{(x+1)t+yt^{2}}\\
 &=& \left(\sum_{k=0}^{\infty}\mathbb{}_{H}{M}_{k}^{(r)}(x,y;k,a,b,c;\bar{\alpha_{r}})\frac{t^{k}}{k!} \right)\left( \sum_{n=0}^{\infty}(\log c)^{n}\frac{t^{n}}{n!}\right).
\end{eqnarray*}
By using Cauchy product and equating the coefficients of $t^{n}$ on both sides of last equation, yields \eqref{19}.
\end{proof}

\begin{theorem}\hfill\\
For arbitrary real or complex parameter (r), the following implicit summation formula holds true
\begin{equation}\label{20}
 \sum_{k=0}^{n}\binom{n}{k}(\log ab)^{k} r^{k} \mathbb{}_{H}{M}_{n}^{(r)}(-x,y;k,a,b,c;\bar{\alpha_{r}})=(-1)^{n-rk}\mathbb{}_{H}{M}_{n}^{(r)}(x,y;k,a,b,c;\bar{\alpha_{r}}).
\end{equation}
\end{theorem}
\begin{proof}
From \eqref{11}, we have
\begin{equation*}
 \sum_{n=0}^{\infty}[1-(-1)^{n}]\mathbb{}_{H}{M}_{n}^{(r)}(x,y;k,a,b,c;\bar{\alpha_{r}})\frac{t^{n}}{n!} = \frac{(-1)^{r}2^{r(1-k)t^{rk}}}{\prod\limits_{i=0}^{r-1}(\alpha_{i}b^{t}-a^{t})}c^{xt+yt^{2}}-
  \frac{(-1)^{r}2^{r(1-k)(-t)^{rk}}}{\prod\limits_{i=0}^{r-1}(\alpha_{i}b^{-t}-a^{-t})}c^{-xt+yt^{2}}
\end{equation*}
\begin{equation*}
 = c^{yt^{2}} \left[\frac{(-1)^{r}2^{r(1-k)t^{rk}}}{\prod\limits_{i=0}^{r-1}(\alpha_{i}b^{t}-a^{t})}c^{xt}-(-1)^{rk} (ab)^{rt}
  \frac{(-1)^{r}2^{r(1-k)(t)^{rk}}}{\prod_{i=0}^{r-1}(\alpha_{i}a^{t}-b^{t})}c^{-xt}\right] \;\;\;\;\;\;\;\;\;\;\;\;\;\;\;\;\;\;\;\;\;\;\;\;\;\;
\end{equation*}
\begin{equation*}
\;\;\;\;\;\; =\sum_{n=0}^{\infty}\mathbb{}_{H}{M}_{n}^{(r)}(x,y;k,a,b,c;\bar{\alpha_{r}})\frac{t^{n}}{n!}-(-1)^{rk}\left(\sum_{k=0}^{\infty} (\log a\;b)^{k} \frac{rt^{k}}{k!} \right)\\
 \sum_{n=0}^{\infty}\mathbb{}_{H}{M}_{n}^{(r)}(-x,y;k,a,b,c;\bar{\alpha_{r}})\frac{t^{n}}{n!}\\
\end{equation*}
\begin{equation*}
 =\sum_{n=0}^{\infty}\mathbb{}_{H}{M}_{n}^{(r)}(x,y;k,a,b,c;\bar{\alpha_{r}})\frac{t^{n}}{n!}-(-1)^{rk}\sum_{n=0}^{\infty}\sum_{k=0}^{n} (\log a\;b)^{k} r^{k}\;\\
  \mathbb{ M}_{n-k}^{(r)}(-x,y;k,a,b,c;\bar{\alpha_{r}})\frac{t^{n}}{n!}.
\end{equation*}

Equating the coefficients of $t^{n}$ on both sides, yields \eqref{20}.
\end{proof}

\begin{theorem}\hfill\\
Let $a,\;b $ and $c$ be positive integers by $a\neq b.$ For $x,y\in \mathbb{R}$ and $n\geq 0.$ Then we have
\begin{equation}\label{23}
 \mathbb{}_{H}{M}_{n}^{(r)}(x+r,y;k,a,b,c;\bar{\alpha_{r}})=\sum_{k=0}^{[\frac{n}{2}]}\binom{n}{2k} y^{k}(\log c)^{k} \mathbb{ M}_{n-2k}^{(r)}(x;k,\frac{a}{c},\frac{b}{c},c;\bar{\alpha_{r}}),
\end{equation}
where $ [ . ]$ is Gauss notation, and represents the maximum integer which does not exceed a number in the square brackets.
\end{theorem}
\begin{proof}
From \eqref{11}, we have
\begin{eqnarray*}
  \sum_{n=0}^{\infty}\mathbb{}_{H}{M}_{n}^{(r)}(x+r,y;k,a,b,c;\bar{\alpha_{r}})\frac{t^{n}}{n!}
  &=& \frac{(-1)^{r}2^{r(1-k)}t^{rk}}{\prod\limits_{i=0}^{r-1}(\alpha_{i}(\frac{b}{c})^{t}-(\frac{a}{c})^{t})}c^{xt}c^{yt^{2}} \\
   &=& \sum_{n=0}^{\infty}\sum_{k=0}^{[\frac{n}{2}]}\binom{n}{2k}(\log c)^{k} y^{k}M_{n-2k}^{(r)}(x;k,\frac{a}{c},\frac{b}{c},c;\bar{\alpha_{r}})\frac{t^{n}}{n!}.
\end{eqnarray*}
Equating the coefficients of $t^{n}$, yields \eqref{23}.
\end{proof}
\begin{theorem}\hfill\\
Explicit formula of  the generalized Apostol Hermite-Genocchi polynomials is given by
\begin{equation}\label{E1}
 \mathbb{}_{H}{M}_{n}^{(r)}(\alpha,\beta;\bar{\alpha_{r}})=2^{r}\sum_{x_{1},...,x_{n}=0}\prod_{i=1}^{r}\;\; (\alpha_{_{i-1}})^{x_{i}}  \sum_{k=0}^{[\frac{n}{m}]}\frac{\beta^{k}n!}{k!(n-mk)!}(\alpha+X)^{n-mk},
\end{equation}
where $X=x_{1}+x_{2}+...+x_{r}.$
\end{theorem}
\begin{proof}
Put $ x=\alpha, \;\; y=\beta, \;\; b=c=e,\;\; a=1,\;\;k=0\;\; and\;\; h(t,\beta)=\beta t^{m}\;\; in \eqref{10} $ we have:
\begin{eqnarray*}
  \sum_{n=0}^{\infty} \mathbb{}_{H}{M}_{n}^{(r)}(\alpha,\beta;\bar{\alpha_{r}})\frac{t^{n}}{n!}
   &=& 2^{r}\;\sum_{x_{1},...,x_{r}=0}\;\;\prod_{i=1}^{r}\;\; (\alpha_{_{i-1}})^{x_{i}}\;\; e^{(x_{1}+...+x_{r}+\alpha)t+\beta t^{m}}.
\end{eqnarray*}
Let $x_{1}+x_{2}+....+x_{r}=X$ ,we have:
\begin{equation*}
    \sum_{n=0}^{\infty} \mathbb{}_{H}{M}_{n}{(r)}(\alpha,\beta;\bar{\alpha_{r}})\frac{t^{n}}{n!} =2^{r}\;\sum_{x_{1},...,x_{r}=0}\;\prod_{i=1}^{r}\;\; (\alpha_{_{i-1}})^{x_{i}}\;\; e^{(X+\alpha)t+\beta t^{m}},\\
\end{equation*}
but the generating function of the generalized Hermite polynomial is given by
\begin{equation*}
  e^{\alpha t+\beta t^{m}}=\sum_{n=0}^{\infty} H_{n,m}(\alpha,\beta)\frac{t^{n}}{n!},
\end{equation*}
then
\begin{equation*}
    \sum_{n=0}^{\infty}\mathbb{}_{H}{M}_{n}^{(r)}(\alpha,\beta;\bar{\alpha_{r}})\frac{t^{n}}{n!} =2^{r}\;\sum_{n=0}^{\infty}[\sum_{x_{1},...,x_{r}=0}\;\prod_{i=1}^{r}\;\; (\alpha_{_{i-1}})^{x_{i}}\;\;H_{n,m}(\alpha+X,\beta)]\frac{t^{n}}{n!}.\\
\end{equation*}
Equating the coefficients $t^{n}$ on the both sides and form the definition of generalized Hermite polynomials, we have
\begin{equation*}
  H_{n,m}(\alpha+X,\beta)= \sum_{k=0}^{[\frac{n}{m}]}\frac{\beta^{k}n!}{k!(n-mk)!}(\alpha+X)^{n-mk},
\end{equation*}
this yields \eqref{E1}.
\end{proof}
\section{General symmetric identity for the generalized Apostol Hermite-Genocchi polynomials}
\begin{theorem}\hfill\\
Let $a,\;b $ and $c$ be positive integers by $a\neq b.$ For $x,y\in \mathbb{R}$ and $n\geq 0.$ Then we have
\begin{equation*}
 \sum_{m=0}^{n}\binom{n}{m} a^{n-m} b^{m}\; \mathbb{}_{H}{M}_{n-m}^{(r)}(bx,b^{2}y;k,A,B,c;\bar{\alpha_{r}})\; \mathbb{}_{H}{M}_{m}^{(r)}(ax,a^{2}y;k,A,B,c;\bar{\alpha_{r}})
 \end{equation*}
 \begin{equation}\label{26}
 = \sum_{m=0}^{n}\binom{n}{m} b^{n-m} a^{m}\; \mathbb{}_{H}{M}_{n-m}^{(r)}(ax,a^{2}y;k,A,B,c;\bar{\alpha_{r}})\; \mathbb{}_{H}{M}_{m}^{(r)}(bx,b^{2}y;k,A,B,c;\bar{\alpha_{r}}).
\end{equation}
\end{theorem}
\begin{proof}
Let
\begin{equation*}
  G(t)=\frac{((-1)^{r}2^{r(1-k)}t^{rk})^{2}}{(\prod\limits_{i=0}^{r-1}(\alpha_{i}A^{at}-B^{at}))
  (\prod\limits_{i=0}^{r-1}(\alpha_{i}A^{bt}-B^{bt}))}\; c^{abxt+a^{2}b^{2}yt^{2}}.
\end{equation*}
Then the expression for $G(t)$ is symmetric in $a$ and $b$ and we can expand $G(t)$ into series in two ways\\
Firstly
\begin{eqnarray*}
  G(t) &=& \frac{1}{(ab)^{rk}}\left(\frac{(-1)^{r}2^{r(1-k)}(at)^{rk}}{\prod\limits_{i=0}^{r-1}(\alpha_{i}A^{at}-B^{at})}\right)\; c^{abxt+a^{2}b^{2}yt^{2}}\left(\frac{(-1)^{r}2^{r(1-k)}(bt)^{rk}}{\prod\limits_{i=0}^{r-1}(\alpha_{i}A^{bt}-B^{bt})}\right)\; c^{abxt+a^{2}b^{2}yt^{2}}\\
 &=& \frac{1}{(ab)^{rk}}\sum_{n=0}^{\infty} \mathbb{}_{H}{M}_{n}^{(r)}(bx,b^{2}y;k,A,B,c;\bar{\alpha_{r}})\frac{(at)^{n}}{n!}\sum_{m=0}^{\infty} \mathbb{}_{H}{M}_{n}^{(r)}(ax,a^{2}y;k,A,B,c;\bar{\alpha_{r}})\frac{(bt)^{m}}{m!}\\
\end{eqnarray*}
\begin{equation}\label{123}
G(t)= \frac{1}{(ab)^{rk}}\sum_{n=0}^{\infty}\sum_{m=0}^{n}\binom{n}{m} a^{n-m}\;b^{m}\;\mathbb{}_{H}{M}_{n}^{(r)}(bx,b^{2}y;k,A,B,c;\bar{\alpha_{r}})\;\mathbb{}_{H}{M}_{n}^{(r)}(ax,a^{2}y;k,A,B,c;\bar{\alpha_{r}})\frac{(t)^{n}}{n!}.
\end{equation}
Secondly
\begin{equation}\label{125}
  G(t)=\frac{1}{(ab)^{rk}}\sum_{n=0}^{\infty}\sum_{m=0}^{n}\binom{n}{m} b^{n-m}\;a^{m}\;\mathbb{}_{H}{M}_{n}^{(r)}(ax,a^{2}y;k,A,B,c;\bar{\alpha_{r}})\;\mathbb{}_{H}{M}_{n}^{(r)}(bx,b^{2}y;k,A,B,c;\bar{\alpha_{r}})\frac{(t)^{n}}{n!}.
\end{equation}
Form Eq. \eqref{123} and Eq. \eqref{125}, by comparing the coefficients of $t^{n}$ on the both sides, yields \eqref{26}.
\end{proof}
\begin{corollary}
Setting $b=1$ in Theorem $5.1$ gives the following result
\begin{equation*}
 \sum_{m=0}^{n}\binom{n}{m} a^{n-m}\; \mathbb{}_{H}{M}_{n-m}^{(r)}(x,y;k,A,B,c;\bar{\alpha_{r}})\; \mathbb{}_{H}{M}_{m}^{(r)}(ax,a^{2}y;k,A,B,c;\bar{\alpha_{r}})
 \end{equation*}
 \begin{equation}\label{27}
 = \sum_{m=0}^{n}\binom{n}{m} a^{m}\; \mathbb{}_{H}{M}_{n-m}^{(r)}(ax,a^{2}y;k,A,B,c;\bar{\alpha_{r}})\; \mathbb{}_{H}{M}_{m}^{(r)}(x,y;k,A,B,c;\bar{\alpha_{r}}).
\end{equation}
\end{corollary}

\section{Application involving probability distribution and some concepts of reliability}
\begin{definition}\hfill\\
Let $ X_{1},X_{2},...,X_{r}$ be nonnegative random variable. Then $\textbf{\underline{X}}=(X_{1},X_{2},...,X_{r})$ is said to generalized Hermite-Genocchi distribution, if its probability mass function is

\begin{equation}\label{P1}
 P(x_{1},x_{2},...,x_{r})=2^{r}\;\mathbf{B}\prod_{i=1}^{r}(\alpha_{_{i-1}})^{X_{i}}\;\;H_{n,m}\left(\sum_{i=1}^{r} x_{i}+\gamma,\beta\right).
\end{equation}
$$ \beta,\gamma \geq0 ;\, r\geq1,\; n,m\in N_{0},$$
where normalizing constant is given by
$$\frac{1}{\mathbf{B}}
=\mathbb{}_{H}{M}_{n}^{(r)}(\gamma,\beta;\bar{\alpha_{_{r}}})=\sum_{\ell_{1},\ell_{2},...,\ell_{r}=0}^{\infty}\prod_{i=1}^{r}(\alpha_{i-1})^{\ell_{i}}H_{n,m}\left(\sum_{i=1}^{r} \ell_{i}+\gamma,\beta\right),$$
and
$$ H_{n,m}(\sum_{i=1}^{r} x_{i}+\gamma,\beta)=\sum_{k=0}^{[\frac{n}{m}]}\frac{\beta^{k}}{k!}\frac{n!}{(n-mk)!}\left(\sum_{i=1}^{r} x_{i}+\gamma\right)^{n-mk}.$$
\end{definition}
$ \mathbb{}_{H}{M}_{n}^{(r)}(\gamma,\beta;\bar{\alpha_{_{r}}})$ is convergent and positive for $ \bar{\alpha_{_{r}}}=(\alpha_{0},\alpha_{1},...,\alpha_{r-1})$  and the distribution is denoted by $GHG[m;\bar{\alpha_{_{r}}};\gamma,\beta]$.

\begin{definition}\hfill\\
Let $m=2$ in \eqref{P1}, then $ \underline{X}$ is said to have Hermite-Genocchi distribution, if its probability mass function is

\begin{equation}
 P(x_{1},x_{2},...,x_{r})=2^{r}\;\mathbf{B}\prod_{i=1}^{r}(\alpha_{_{i-1}})^{X_{i}}\;\;H_{n,2}\left(\sum_{i=1}^{r} x_{i}+\gamma,\beta\right),
\end{equation}
where
$$ H_{n,2}(\sum_{i=1}^{r} x_{i}+\gamma,\beta)=\sum_{k=0}^{[\frac{n}{2}]}\frac{\beta^{k}}{k!}\frac{n!}{(n-2k)!}\left(\sum_{i=1}^{r} x_{i}+\gamma\right)^{n-2k}.$$
\end{definition}
The distribution is denoted by $HG[\bar{\alpha_{_{r}}};\gamma,\beta]$.

\textbf{Result 5.1:}
Let $ \underline{X}=(X_{1},...,X_{r})$ follow the generalized Hermite-Genocchi distribution, then the probability generating function of $\underline{X}$ is the following
\begin{equation}\label{p1}
 G_{\underline{X}}(t)=\mathbf{B}\; \mathbb{}_{H}{M}_{n}^{(r)}(\gamma,\beta;\mathbf{t}\bar{\alpha_{_{r}}})
\end{equation}
\begin{proof}
Form the definition of the probability generating function, we have
\begin{equation*}
   G_{\underline{X}}(t) = E(\mathbf{t}^{\underline{X}})=\sum_{x_{1},x_{2},...,x_{r}=0}P(x_{1},x_{2},...,x_{r}) \;\mathbf{t}^{^{\underline{X}}},
\end{equation*}
where $ \mathbf{t}=(t_{1},t_{2},...,t_{r}),$
this yield \eqref{p1}.
\end{proof}

\textbf{Result 5.2:}
The moment generating function of the generalized Apostol Hermite-Genocchi distribution is the following
\begin{equation}\label{p2}
 M_{\underline{X}}(t)=\mathbf{B}\; \mathbb{}_{H}{M}_{n}^{(r)}(\gamma,\beta;e^{\mathbf{t}}\bar{\alpha_{_{r}}})
\end{equation}

\textbf{Result 5.3:}
The $\ell-th$ factorial moments $\mu_{_{[\ell]}} $ with moment generating function $ M_{\underline{X}}(t)$ of the generalized Hermite-Genocchi distribution is the following
\begin{equation}\label{p3}
 \mu_{[\ell]}=2^{r}\; \mathbf{B}\; \sum^{\infty}_{x_{1},x_{2},...,x_{r}=0}\; ( x_{i})_{\underline{\ell}} \prod_{i=1}^{r}(\alpha_{i-1})^{x_{i}}\;\;H_{n,m}\left(\sum_{i=1}^{r} x_{i}+\gamma,\beta\right).
\end{equation}

\textbf{Result 5.4:}
The $\ell-th$ moments $\grave{\mu_{_{\ell}}}$  with moment generating function $ M_{\underline{X}}(t)$ of the generalized Hermite-Genocchi distribution is the following
\begin{equation}\label{p4}
\grave{ \mu_{\ell}}=2^{r}\;\mathbf{B}\;\sum^{\infty}_{x_{1},x_{2},...,x_{r}}=0\; x_{i}^{\ell}\;\prod_{i=1}^{r}(\alpha_{i-1})^{x_{i}}\;\;H_{n,m}\left(\sum_{i=1}^{r} x_{i}+\gamma,\beta\right).
\end{equation}
\textbf{Result 5.5:}
The mean and variance of the generalized Apostol Hermite-Genocchi distribution with $\ell-th$ moments $\grave{\mu_{_{\ell}}}$ is the following
\begin{equation}\label{p5}
E(\underline{X})=2^{r}\;\mathbf{B}\sum^{\infty}_{x_{1},x_{2},...,x_{r}=0}\; x_{i}\;\prod_{i=1}^{r}(\alpha_{_{i-1}})^{x_{i}}\;\;H_{n,m}\left(\sum_{i=1}^{r} x_{i} +\gamma,\beta\right).
\end{equation}
\begin{eqnarray}\label{p6}
Var(\underline{X})&=&\left[2^{r}\;\mathbf{B}\sum^{\infty}_{x_{1},x_{2},...,x_{r}=0}\; x_{i}^{2}\;\prod_{i=1}^{r}(\alpha_{_{i-1}})^{x_{i}}\;\;H_{n,m}\left(\sum_{i=1}^{r} x_{i} +\gamma,\beta\right)\right]\\ \nonumber
& - &\left[2^{r}\;\mathbf{B}\sum^{\infty}_{x_{1},x_{2},...,x_{r}=0}\; x_{i}\;\prod_{i=1}^{r-1}(\alpha_{_{i-1}})^{x_{i}}\;\;H_{n,m}\left(\sum_{i=1}^{r} x_{i} +\gamma,\beta\right)
\right]^{2}.
\end{eqnarray}
\begin{lemma}
Let $X_{i}\sim GHG(m,\alpha_{i},\gamma,\beta),\;\; i = 1, 2, ..., r-1$ Then the marginal joint
cumulative distribution function
\begin{equation}\label{p7}
  P(X_{i}\leq x_{i})=1-\mathbf{B}\; (\alpha_{i-1})^{x_{i}}\;\mathbb{}_{H}{M}_{n}^{(r)}(\gamma+x_{i}+1,\beta;\bar{\alpha_{_{r}}})
\end{equation}
\end{lemma}
\begin{proof}
since
\begin{eqnarray*}
  P(X_{i}\leq x_{i}) &=& 1- P(X_{i} > x_{i})\\
   &=&1-2^{r}\;\mathbf{B}\;\sum^{\infty}_ {\ell_{i}=x_{i}+1}\sum^{\infty}_{ \substack {\ell_{1},...,\ell_{i-1}\\ \ell_{i+1},...,\ell_{r}=0}}(\alpha_{0})^{\ell_{1}}(\alpha_{1})^{\ell_{2}}...(\alpha_{i-1})^{\ell_{i}}...(\alpha_{r-1})^{\ell_{r}}
   H_{n,m}\left(\sum_{i=1}^{r}\ell_{i} +\gamma,\beta\right).
\end{eqnarray*}
Setting $\ell_{i}-x_{i}-1=\ell,$ then
\begin{eqnarray*}
  P(X_{i}\leq x_{i})
   &=&1-2^{r}\;\mathbf{B}\;\sum^{\infty}_{ \substack {\ell,\ell_{1},...,\ell_{i-1}\\ \ell_{i+1},...,\ell_{r}=0}}(\alpha_{0})^{\ell_{1}}(\alpha_{1})^{\ell_{2}}...(\alpha_{i-1})^{x_{i}+\ell_{i}}+1...(\alpha_{r-1})^{\ell_{r}}
   H_{n,m}\left(\sum_{i=1}^{r}\ell_{i} +\gamma+x_{i}+1,\beta\right).
\end{eqnarray*}
From \eqref{P1}, obtain \eqref{p7}.
\end{proof}

\begin{theorem}\hfill\\
Suppose $ X_{1}, X_{2}, ..., X_{r} $ are mutually independent where $X_{i}\sim GHG(m,\alpha_{i},\gamma,\beta), i = 1, 2, ..., r-1$. Then the multivariate cumulative distribution function is given by
\begin{equation}\label{p8}
P(X_{1}\leq x_{1},X_{2}\leq x_{2},...,X_{r}\leq x_{r})=\prod_{i=1}^{r}\left(1-\mathbf{B}\; (\alpha_{i-1})^{x_{i}}\;\mathbb{}_{H}{M}_{n}^{(r)}(\gamma+x_{i}+1,\beta;\bar{\alpha_{_{r}}})
    \right).
\end{equation}
\end{theorem}
\begin{proof}
Since $ X_{1}, X_{2}, ..., X_{r} $ are mutually independent, then
\begin{eqnarray*}
 P(X_{1}\leq x_{1},X_{2}\leq x_{2},...,X_{r}\leq x_{r})
  &=&\prod_{i=1}^{r} \left(1- P(X_{i} > x_{i})\right).
\end{eqnarray*}
From \eqref{p7}, we obtain \eqref{p8}.
\end{proof}

\subsection{Reliability concepts for GHG distributions}
\subsubsection{ Multivariate reliability function}
\begin{theorem}\hfill\\
Let $\textsc{x} = (x_{1}, x_{2}, ..., x_{r})\in \mathbf{R}^{r}_{+}$
representing the lifetimes of r−component system with the multivariate reliability function
\begin{equation}\label{p9}
R(\textsc{x})=\mathbf{B}\;\prod_{i=1}^{r} (\alpha_{i-1})^{x_{i}}\;\mathbb{}_{H}{M}_{n}^{(r)}(\gamma+x_{1}+x_{2}+...+x_{r},\beta;\bar{\alpha_{_{r}}}).
\end{equation}
\end{theorem}
\begin{proof}
From the definition of the multivariate reliability function, see \cite{Nair} and \eqref{P1}
\begin{eqnarray*}
  R(x_{1},x_{2},...,x_{r}) &=&  P(X_{1}\geq x_{1},X_{2}\geq x_{2},...,X_{r}\geq x_{r} ) \\
  &=&\sum_{m_{1}\geq x_{1}}\sum_{m_{2}\geq x_{2}}...\sum_{m_{r}\geq x_{r}} P(X_{1}= m_{1},X_{2}= m_{2},...,X_{r}= m_{r} ).
\end{eqnarray*}
We obtain \eqref{p9}.
\end{proof}
\begin{theorem}
$R(\textsc{x})$ is said to be
\begin{description}
  \item[i)]Multivariate new better than used (Multivariate new worse than used ) MNBU
(MNWU) if
\begin{equation}\label{p10}
 \mathbb{}_{H}{M}_{n}^{(r)}(\gamma,\beta;\bar{\alpha_{_{r}}})\mathbb{}_{H}{M}_{n}^{(r)}(\mathbf{x_{r}}+\acute{t}+\gamma,\beta;\bar{\alpha_{_{r}}})
\leq (\geq)
 \mathbb{}_{H}{M}_{n}^{(r)}(\mathbf{x_{r}}+\gamma,\beta;\bar{\alpha_{_{r}}})\mathbb{}_{H}{M}_{n}^{(r)}(\grave{t}+\gamma,\beta;\bar{\alpha_{_{r}}})
\end{equation}
  \item[ii)]Multivariate new better than used in expectation (Multivariate new worse than
used in expectation ) MNBUE (MNWUE) if
\begin{multline}\label{p11}
 \mathbb{}_{H}{M}_{n}^{(r)}(\gamma,\beta;\bar{\alpha_{_{r}}})\sum_{t_{1},t_{2},...,t_{r}=0}\prod^{r-1}_{i=0}\alpha_{i}\;  \mathbb{}_{H}{M}_{n}^{(r)}(\mathbf{x_{r}}+\acute{t}+\gamma,\beta;\bar{\alpha_{_{r}}})\leq (\geq)\\ \mathbb{}_{H}{M}_{n}^{(r)}(\mathbf{x_{r}}+\gamma,\beta;\bar{\alpha_{_{r}}})\sum_{t_{1},t_{2},...,t_{r}=0}\prod^{r-1}_{i=0}\alpha_{i}\;  \mathbb{}_{H}{M}_{n}^{(r)}(\acute{t}+\gamma,\beta;\bar{\alpha_{_{r}}}),
\end{multline}

\end{description}
where $\mathbf{x_{r}} = x_{1}+ x_{2}+ ...+ x_{r},$ $\acute{t}=t_{1}+t_{2}+...+t_{r}.$
\end{theorem}
\begin{proof}
If
\begin{eqnarray*}
\mathbb{}_{H}{M}_{n}^{(r)}(\mathbf{x_{r}}+\acute{t}+\gamma,\beta;\bar{\alpha_{_{r}}})
&\leq&\frac{\mathbb{}_{H}{M}_{n}^{(r)}(\mathbf{X_{r}}+\gamma,\beta;\bar{\alpha_{_{r}}})\mathbb{}_{H}{M}_{n}^{(r)}(\acute{t}+\gamma,\beta;\bar{\alpha_{_{r}}})}
{\mathbb{}_{H}{M}_{n}^{(r)}(\gamma,\beta;\bar{\alpha_{_{r}}})}\\
 \frac{\mathbb{}_{H}{M}_{n}^{(r)}(\mathbf{x_{r}}+\acute{t}+\gamma,\beta;\bar{\alpha_{_{r}}})
}{\mathbb{}_{H}{M}_{n}^{(r)}(\gamma,\beta;\bar{\alpha_{_{r}}})} &\leq&
\frac{\mathbb{}_{H}{M}_{n}^{(r)}(\mathbf{x_{r}}+\gamma,\beta;\bar{\alpha_{_{r}}})
}{\mathbb{}_{H}{M}_{n}^{(r)}(\gamma,\beta;\bar{\alpha_{_{r}}})}
\frac{\mathbb{}_{H}{M}_{n}^{(r)}(\acute{t}+\gamma,\beta;\bar{\alpha_{_{r}}})
}{\mathbb{}_{H}{M}_{n}^{(r)}(\gamma,\beta;\bar{\alpha_{_{r}}})}  \\
 \prod_{i=1}^{r-1}(\alpha_{i-1})^{x_{i}+t_{i}} \frac{\mathbb{}_{H}{M}_{n}^{(r)}(\mathbf{x_{r}}+\acute{t}+\gamma,\beta;\bar{\alpha_{_{r}}})
}{\mathbb{}_{H}{M}_{n}^{(r)}(\gamma,\beta;\bar{\alpha_{_{r}}})} &\leq&
\prod_{i=1}^{r-1}(\alpha_{i-1})^{x_{i}}\frac{\mathbb{}_{H}{M}_{n}^{(r)}(\mathbf{x_{r}}+\gamma,\beta;\bar{\alpha_{_{r}}})
}{\mathbb{}_{H}{M}_{n}^{(r)}(\gamma,\beta;\bar{\alpha_{_{r}}})}\\
& &\prod_{i=1}^{r-1}(\alpha_{i-1})^{t_{i}}\frac{\mathbb{}_{H}{M}_{n}^{(r)}(\acute{t}+\gamma,\beta;\bar{\alpha_{_{r}}})
}{\mathbb{}_{H}{M}_{n}^{(r)}(\gamma,\beta;\bar{\alpha_{_{r}}})}
\end{eqnarray*}
hence, we get
\begin{equation*}
  R(x_{1}+t_{1},x_{1}+t_{1},...,x_{1}+t_{1})\leq R(x_{1},x_{2},...,x_{r})R(t_{1},t_{2},...,t_{r})
\end{equation*}
From the definition of the Multivariate new better used, see \cite{Han}, then $R(\textsc{x})$ is MNBU.
Similarly, from the definition of MNBUE (multivariate new better used in expectation), see \cite{Han}, then $R(\textsc{x})$ is MNBUE.
\end{proof}
\subsubsection{Multivariate hazard rate function}
Since multivariate hazard rate function is defined as, see \cite{Nair}\\
$$ h(x)=\left( h_{1}(x),h_{2}(x),...,h_{r}(x)\right)$$
\begin{eqnarray*}
  h_{i} &=& P(X_{i})=x_{i}|X\geq \textsc{x}) \\
  &=& 1-\frac{R(x_{1},x_{2},...,x_{i-1},x_{i}+1,x_{i-1},...,x_{r})}{R(x_{1},x_{2},...,x_{r})},
\end{eqnarray*}
where $ \textsc{x}=(x_{1},x_{2},...,x_{r})\in\mathbf{R}^{r}_{+}$, so we obtain the following theorems
\begin{theorem}\hfill\\
The multivariate hazard rate function of GHG distribution is given
by
\begin{equation}\label{p12}
 h_{i}(x)=1-\alpha_{i-1} \frac{\mathbb{}_{H}{M}_{n}^{(r)}(x_{1}+x_{2}+...+x_{i-1}+x_{i}+1+x_{i-1}+...+x_{r}+\gamma,\beta;\bar{\alpha_{_{r}}})}
 {\mathbb{}_{H}{M}_{n}^{(r)}(x_{1}+x_{2}+...+x_{r}+\gamma,\beta;\bar{\alpha_{_{r}}})}.
\end{equation}
\end{theorem}
\begin{theorem}\hfill\\
Let $\mathbf{\grave{X_{r}}}=(x_{1}+x_{2}+...+x_{i-1}+x_{i}+1+x_{i+1}+...+x_{r}),\; \mathbf{x_{r}}=(x_{1}+x_{2}+...+x_{r})$ and
$ \acute{t}=t_{1}+t_{2}+...+t_{r},$
then the following statement about GHG distribution holds.\\
The GHG distribution is multivariate increasing hazard rate (multivariate decreasing hazard rate) MIHR (MDHR) iff
\begin{equation}\label{p13}
  \frac{\mathbb{}_{H}{M}_{n}^{(r)}(\grave{\mathbf{X_{r}}}+\acute{t}+\gamma,\beta;\bar{\alpha_{_{r}}})}{\mathbb{}_{H}{M}_{n}^{(r)}(\mathbf{x_{r}}+\mathbf{t}+\gamma,\beta;\bar{\alpha_{_{r}}})}
  \leq(\geq)\frac{\mathbb{}_{H}{M}_{n}^{(r)}(\grave{\mathbf{X_{r}}}+\gamma,\beta;\bar{\alpha_{_{r}}})}{\mathbb{}_{H}{M}_{n}^{(r)}(\mathbf{x_{r}}+\gamma,\beta;\bar{\alpha_{_{r}}})}.
  \end{equation}
\end{theorem}

\begin{proof}
If
\begin{equation*}
    \frac{\mathbb{}_{H}{M}_{n}^{(r)}(\grave{\mathbf{X_{r}}}+\acute{t}+\gamma,\beta;\bar{\alpha_{_{r}}})}{\mathbb{}_{H}{M}_{n}^{(r)}(\mathbf{x_{r}}+\mathbf{t}+\gamma,\beta;\bar{\alpha_{_{r}}})}
    \leq \frac{\mathbb{}_{H}{M}_{n}^{(r)}(\grave{\mathbf{X_{r}}}+\gamma,\beta;\bar{\alpha_{_{r}}})}{\mathbb{}_{H}{M}_{n}^{(r)}(\mathbf{x_{r}}+\gamma,\beta;\bar{\alpha_{_{r}}})},
\end{equation*}
then
\begin{equation*}
 1-\alpha_{i-1}   \frac{\mathbb{}_{H}{M}_{n}^{(r)}(\grave{\mathbf{X_{r}}}+\acute{t}+\gamma,\beta;\bar{\alpha_{_{r}}})}{\mathbb{}_{H}{M}_{n}^{(r)}(\mathbf{x_{r}}+\mathbf{t}+\gamma,\beta;\bar{\alpha_{_{r}}})}
    \geq 1-\alpha_{i-1}  \frac{\mathbb{}_{H}{M}_{n}^{(r)}(\grave{\mathbf{X_{r}}}+\gamma,\beta;\bar{\alpha_{_{r}}})}{\mathbb{}_{H}{M}_{n}^{(r)}(\mathbf{x_{r}}+\gamma,\beta;\bar{\alpha_{_{r}}})},
\end{equation*}
hence
$$h_{i}(x_{1}+t_{1},...,x_{i}+t_{i},...,x_{r}+t_{r})\geq h_{i}(x_{1},x_{2},...,x_{r}).$$

Therefore GHG distribution is MIHR (multivariate increasing Hazard rate).
\end{proof}

\textbf{References}
\bibliographystyle{plain}

\begin{thebibliography}{99}
\bibitem{araci}
Araci, S., Khan, W. A., Acikgoz, M., \"{O}zel, C. and Kumam, P. (2016)
 \emph{A new generalization of Apostol type Hermite-Genocchi polynomials and its applications,} SpringerPlus 5(1): 1-17.


\bibitem{bell}
Bell, E. T. (1934)
 \emph{Exponential polynomials,} Ann. Math. 53 : 258-277.


\bibitem{datt}
Dattoli, G., Lorenzutta, S. and Cesarano, C. (1999)
 \emph{Finit sums and generalized forms of Bernoulli polynomials,} Rendicont di. Math. 19 : 385-391.

\bibitem{bs}
El-Desouky, B. S. and Gomaa, R. S. (2014)
\emph{A new unified family of generalized Apostol-Euler, Bernoulli and Genocchi polynomials,} Appl. Math. Comput. 247 : 695-702.

\bibitem{kurt}
Geboury, S. and Kurt, B. (2012)
 \emph{Some relations involving Hermite-based Apostol-Genocchi polynomials,} Appl. Math. Sci. 6(82) : 4091-4102.

\bibitem{Gupta}
 Gupta, P.L., Gupta, R.C.  and Tripathi, R.C. (1997)
 \emph{On the monotonic properties of
the discrete failure rates,} Journal of statistical planning and inference, 65 : 255-268.

\bibitem{Han}
 Hanagal, D. D. (1998)
 \emph{Testing whether the survival function is multivariate new better
than used,} Statistical Papers, 39 : 203-211.

\bibitem{jolany}
Jolany, H., Sharifi, H. and Aliklaye, R. E. (2013)
 \emph{Some results for the Apostol-Genocchi polynomials of higher order,} Bull Malays. Math. Sci. Soc. 2 : 465-479.

\bibitem{lu1}
Luo, Q. M. (2006)
 \emph{Apostol Euler polynomials of higher order an Gaussian hypergeometric function,} Taiwanese J. Math. 10(4) : 917-925.

\bibitem{lu2}
Luo, Q. M. (2009)
 \emph{Fourier expansions and integral representations for the Genocchi polynomials,} J. Integer Seq., 12(1) :1-9.

  \bibitem{lu3}
Luo, Q. M., Srivastava, H. M. (2005)
 \emph{Some generalizations of the Apostol-Bernoulli polynomials and Apostol-Euler polynomials,} J. Math. Anal. Appl. 308(1) : 290-302.

\bibitem{lu4}
Luo, Q. M., Srivastava, H. M. (2006)
 \emph{Some relationships between the Apostol-Bernoulli polynomials and Apostol-Euler polynomials,} Comput. Math. Appl. 51(3) : 631-642.

 \bibitem{lu5}
Luo, Q. M. (2011)
 \emph{Extension for the Genocchi polynomials and their fourier expansions and integral representations,} Qsaaka J. Math. 48(2) : 291-309.



\bibitem{Nair}
 Nair, N. U. and Asha, G. (1997) \emph{Some classes of multivariate life distribution in discrete
time,} J. Multivar. Anal.,62 : 181-189.

\bibitem{Naka}
Nakagawa, T.  and Osaki, S. (1975)
\emph{The discrete Weibull distribution,} IEEE Transactions on Reliability, 24 : 300-301.

\bibitem{patha}
Pathan, M. A. and Khan, W. A. (2015a)
 \emph{Some implicit summation formulas and symmetric identities for the generalizated Hermite-Bernoulli polynomials,} Mediterr. J. Math. 12(3) : 679-695.

 \bibitem{pathb}
Pathan, M. A. and Khan, W. A. (2015b)
 \emph{A new class of generalized polynomials associated with Hermite and Euler polynomials,} Mediterr. J. Math. 13(3) : 913-928.

 \bibitem{pathd}
Pathan, M. A. and Khan, W. A. (2015d)
 \emph{Some new classes of generalized Hermite-based Apostol-Euler and Apostol-Genocchi polynomials,} Fasciculli. Math. Vol. 55(1) : 153-170.

\bibitem{Roy}
 Roy, D. and Gupta, R. P. (1999)
 \emph{Characterizations and model selection through reliability measures in the discrete case,} Statistic and Probability Letters, 43 : 197-206.

\bibitem{Salvia}
 Salvia, A. A. and Bollinger, R. C.(1982)
 \emph{ On discrete hazard functions,} IEEE Transactions on Reliability, 31 : 458-459.

\bibitem{sriv}
Srivastava, H. M. (2000)
 \emph{Some formulas for the Bernoulli and Euler polynomials at rational arguments, }Math. Proc. Camb. Philos. Soc. 129(1) : 77-84.

\bibitem{sriv2}
Srivastava, H. M., Garg. M. and Choudhary. S. (2010)
 \emph{A new generalization of the Bernoulli and related polynomials,} Russian. J. Math. Phys. 17(2) : 251-261.

\bibitem{ornig}
Zornig, P. and Altmann, G.(1995)
\emph{Unified representation of Zipf distributions,} Comput.
Statist. Data. Anal. 19 : 461-473.

\end{thebibliography}

\end{document}